\documentclass[10pt,a4paper]{amsart}
\usepackage{latexsym,amssymb,graphicx,amsmath,amscd}

\usepackage[all,cmtip]{xy}

\newcommand{\CC}{\Bbb C}
\newcommand{\PP}{\Bbb P}

\newtheorem{defi}{Definition}[section]
\newtheorem{pro}[defi]{Proposition}

{

\newtheorem{theo}[defi]{Theorem}

\newtheorem{exa}[defi]{Example}
\newtheorem{exas}[defi]{Examples}

\newtheorem{rem}[defi]{Remark}

\newtheorem{conj}[defi]{Conjecture}
\newtheorem{conjs}[defi]{Conjectures}
\newtheorem{que}[defi]{Question}

\title{Order $1$  congruences of lines with smooth fundamental scheme}
\author{Christian Peskine }
\begin{document}
\begin{abstract} 
 In this note we present a notion of fundamental scheme for Cohen-Macaulay, order $1$, irreducible congruences of lines.
We show that such a congruence is formed by the $k$-secant lines to its fundamental scheme  for a number $k$
that we call the secant index of the congruence. If the fundamental scheme $X$ is a smooth connected variety in ${\Bbb P}^N$,
then $k=(N-1)/(c-1)$ (where $c$ is the codimension of $X$) and $X$ comes equipped with a special tangency divisor cut out by a virtual hypersurface 
of degree $k-2$ (to be precise, linearly equivalent to a section by an hypersurface of degree $(k-2)$ without being cut by one).
This is explained in the main theorem of this paper. This theorem is followed by a complete classification of known locally Cohen-Macaulay 
order $1$ congruences of lines with smooth
fundamental scheme. To conclude we remark that according to Zak's  classification of Severi Varieties and Harthsorne 
conjecture for low codimension varieties, this classification is complete. 
\end{abstract}

\maketitle


\section{Introduction}

As usual ${\PP}={\PP}^{N}={\PP}(V)$, where $V$ is a complex vector space of dimension $(N+1)$, is the complex projective space. 
A congruence of lines of ${\PP}$ is an $(N-1)$-dimensional variety (reduced scheme)  embedded in the Grassmann variety of lines $G(1,N)=G(2,V)$. 

We recall the  following classical notations.

\begin{defi}  Let $\Sigma \subset G(2,V)$ be a congruence of lines.

1) The order $o(\Sigma)$ of $\Sigma$  is the number of lines of $\Sigma$ passing through a general point of ${\PP}$.

2) The (set theoretical) fundamental locus $X( \Sigma )$ of $\Sigma$  is the closed set formed by all points $x\in {\PP}$ through which pass infinitely many lines of $\Sigma$.

\end{defi}

\begin{exas} Let $C \subset {\PP}^3$ be a smooth projective curve of degree $d$ and genus $g$. The $2$-secant lines to $C$ form a  congruence of lines  $\Sigma_2(C)\subset G(1,3)$. The order of this congruence is $o(\Sigma_2(C)) = d(d-3)/2 - (g-1)$ and the (set theoretical) fundamental locus contains $C$.

- The $2$-secant lines to a twisted cubic curve form a congruence of order $1$ whose (set theoretical) fundamental locus is the curve itself.

- The $2$-secant lines to a normal  elliptic curve $C \subset \PP^3$ form a congruence of order $2$. Its fundamental locus is the union of the curve 
itself and of $4$ points outside the curve,  vertices of cones of $2$-secant lines
\end{exas}

\vspace{0,3cm}

To our knowledge the fundamental locus is classically  defined as a set of points. In this note, we claim that  Cohen-Macaulay irreducible congruences of order $1$  have a well defined fundamental scheme  and we state and prove a theorem which hopefully justifies this point of view. 
My interest for order $1$ congruences of lines goes back to Zak's classification of Severi Varieties (see  \cite{Zak}). In particular I have discussed often  with Fyodor Zak about the congruences of $3$-secant lines to the projected Severi Varieties (also studied by Iliev and Manivel,
see (\cite{Il-Ma}). I remember with great pleasure the day,
many years ago, when discussing with F. Han and F. Zak we convinced ourselves that the congruence of $2$-secant lines to a twisted cubic $C$ was in fact the congruence of $4$-secant lines to the full first infinitesimal neighborhood of $C$ and that furthermore quadric hypersurfaces  cut a "non-complete linear system" on this infinitesimal neighborhood. This paper is in many aspects a partial survey of the pleasant
discussions I have had with F. Han (see \cite{Han})  and F. Zak since this discovery. Fyodor Zak has been particularly generous with
his time and his friendly critics. I wish to thank him and Jean Valles for helping me to write down these notes.

\section{Notations and Examples}
Let us start by organizing our notations. To this aim, we recall the Euler complex  on ${\PP}$ 
$$0\rightarrow  \Omega_{\PP^N} \rightarrow  V\otimes O_{\PP}(-1) \rightarrow  O_{\PP} \rightarrow  0,$$
and the  tautological complex of vector bundles on $G(2,V)$
$$0\rightarrow  K^*  \rightarrow  V\otimes O_{G(2,V)} \rightarrow  Q \rightarrow  0,$$
where $Q$ is the canonical quotient rank-$2$ vector bundle on $G(2,V)$.

\vspace{0,3cm}

We denote  by $ {\mathcal  I} \subset G(1,N) \times {\PP}^{N}$ the incidence variety line/point. We recall that
$$q : {\mathcal  I} = {\PP}_{G(1,N)} (Q)\rightarrow G(1,N)$$ is a projective line bundle on the one hand, and that  
$$p : {\mathcal  I} = {\PP}_{{\PP}^{N}}( \Omega_{\PP^N}(2))\rightarrow {\PP}^N$$ is a projective  (${\PP}^{N-1}$)-bundle  on the other hand.

If $x\in {\PP}^N$, we   write $ {\PP}^{N-1}(x)$ for the  fiber  $p^{-1}(x)$  and 
$$\Sigma(x)=p^{-1}(x) \cap q^{-1}(\Sigma)$$
for the scheme  of lines of $\Sigma$ passing through  $x$.
When $o(\Sigma) \neq 0$,  it is the degree of the generically finite morphism $q^{-1}(\Sigma) \rightarrow {\PP}^N$.

As a last set of  notations, we denote the tautological line bundles on $\PP^N$ and   $G(1,N)$ (Pl$\ddot{\mathrm{u}}$cker embedding) by
$$O_{{\PP}^{N}}(1)=O_{{\PP}^{N}}(\theta) \mbox{  and  } O_{G(1,N)}(1) =O_{G(1,N)}(\eta);$$ 
hoping to avoid too many stars, we write 
$$O_{\mathcal  I}({k}\eta, l\theta)= O_{\mathcal  I}({k}\eta + l\theta) =  q^*(O_{G(1,N)}(k\eta))  \otimes_{O_{\mathcal  I}} p^*(O_{{\PP}^{N}}(l\theta)). $$

To conclude this section, we study a list of examples, with a special interest in  $\Sigma(x)$, the  family of lines of $\Sigma$  passing through a (sometimes general) point $x$ of the fundamental locus. Our interest in this "fiber" will be explained and justified when we introduce the  ``fundamental scheme" of the congruence. We present these examples in three separated groups. To be precise the congruences we describe are all congruences of $k$-secant lines to classically known varieties for $4\geq k\geq 2$. We choose, it will be justified later, to organize these examples following the number $k$.

To avoid any misunderstanding, let us begin by being precise about what is a $k$-secant line to a variety.

\begin{defi} Let $X \subset \PP$ be a projective variety and $L \subset \PP$ a line. 
If $L \nsubseteq X$, we say that $L$ is a  $k$-secant line to $X$ if the finite scheme $L \cap X$ has degree $\geq k$. 

The $k$-secant lines to $X$ which are not contained in $X$ form a well defined quasi-projective subscheme of the Grassmann variety (see \cite{Gr-Pe} for example) of lines in $\PP$. The closure $\mathrm{Sec}_k(X)$ in the Grassman Variety of this quasi-projective scheme is the $k$-secant scheme to $X$.
\end{defi}

From this definition, it is clear that if $L\subset X$, then $\{L\} \in \mathrm{Sec}_2(X)$, but $\{L\}$ is not necessarily in $\mathrm{Sec}_k(X)$for $k\geq 3$.  For example a Palatini $3$-fold in $\PP^5$ is ruled over a cubic surface, but the family of  lines of the ruling and $\mathrm{Sec}_4(X)$ are disjoint varieties in the Grassmann Variety.

\begin{exas} 
1) If $C \subset \PP^3$ is a twisted cubic  and $\Sigma_2(C) = \mathrm{Sec}_2(C) \subset G(1,3)$ is the order $1$ congruence of $2$-secant lines to $C$, then 
 $\Sigma_2(C)(x)  \subset  \PP^2(x) \subset G(1,3)$ is a conic for all $x\in C$.

2) If $C=L_1 \cup L_2  \subset \PP^3$ is the disjoint union of two lines  and $\Sigma_2(C) \subset G(1,3)$ is the order $1$ congruence of lines intersecting $L_1$ and $L_2$, then  $\Sigma_2(C)(x) \subset \PP^2(x)  \subset G(1,3)$  is a line for all $x\in C$. Note that $L_1,L_2 \notin \Sigma_2(C)$.

3) Let $X \subset \PP^5$ be a normal rational ruled surface (of degree $4$)  without exceptional line. If $\Sigma_2(X) \subset G(1,5)$ is the order $1$ congruence of $2$-secant lines to $X$, then $\Sigma_2(X)(x) \subset \PP^4(x) \subset G(1,5)$  is a ruled cubic surface for all $x\in X$.
Note here that the lines of the ruling are indeed elements of the congruence $\Sigma_2(X)$.

\end{exas}

We note that in the two first examples $\Sigma_2(C)(x)=\PP^1$  for $x \in C$, but embedds as a conic  in one case and as a line 
in the other case. This difference will be explained by the structure of the fundamental schemes of these two congruences.

The last of these $3$ examples was communicated to me by E. Mezzetti and P. de Poi.

\begin{exas} 

1) If $S  \subset \PP^4$ is a projected smooth Veronese surface and $\Sigma_3(S) \subset G(1,4)$ is the  congruence of $3$-secant lines to $S$, then $o(\Sigma_3(S))=1$ and $\Sigma_3(S)(x) \subset \PP^3(x) \subset G(1,4)$  is a line for all $x\in S$.

2) If $B  \subset \PP^4$ is a Bordiga surface and $\Sigma_3(B) \subset G(1,4)$ is the  congruence of $3$-secant lines to $B$, then $o(\Sigma_3(B))=1$ and $\Sigma_3(B)(x) \subset \PP^3(x)  \subset G(1,4)$  is a twisted cubic for a  \texttt{general}  point $x\in B$.
\end{exas}

We recall that a Bordiga surface in $\PP^4$ is cut out by the $0$-th Fitting ideal of a (general enough) $3\times 4$ matrix with linear coefficients.

\begin{exas} 

1) If $X  \subset \PP^5$ is a Palatini $3$-fold and $\Sigma_4(X) \subset G(1,5)$ is the  congruence of $4$-secant lines to $X$, then $o(\Sigma_4(X))=1$ and $\Sigma_4(X)(x) \subset G(1,5)$  is a line for all $x\in X$.

2) If $Sc \subset \PP^5$ is the scroll over a $K_3$ surface cut out in $G(1,5) $ by a general $\PP^8$ of the Plucker space,  
then the congruence $\Sigma_4(Sc) \subset G(1,5)$ of $4$-secant lines to $Sc$ has order $1$ and 
$\Sigma_4(Sc)(x) \subset G(1,5)$  is a conic for all $x\in Sc$.
\end{exas}

The computation of $\Sigma_4(Sc)(x) \subset G(1,5)$ in this last example was explained to me separately by F. Zak and by P. de Poi and E. Mezzetti.

\section{The scheme structure of the fundamental locus of Cohen-Macaulay, order $1$, irreducible congruences}

From here $\Sigma$ is a  Cohen-Macaulay, order $1$, irreducible congruence of lines. Since $\Sigma$ is irreducible,
so is $q^{-1}(\Sigma)$.

Note that since $\Sigma$ is Cohen-Macaulay, so is $q^{-1}(\Sigma)$ and the finite and birational morphism  $q^{-1}(\Sigma)\setminus p^{-1}(X(\Sigma))  \rightarrow \PP^N \setminus X(\Sigma)$ is flat, hence is an isomorphism. 
Since $\Sigma$ is irreducible, so is $q^{-1}(\Sigma)$ and the fundamental locus $X(\Sigma)$ has codimension at least $2$.

We denote by $J_{\Sigma/G}$ the sheaf of ideals of $\Sigma$ in $G(1,N)$ and we consider the exact sequence
$$ 0 \rightarrow J_{\Sigma/G}(\eta)  \rightarrow O_G (\eta)   \rightarrow O_{\Sigma} (\eta)  \rightarrow 0.$$
 Recalling that $p_*(q^*O_G (\eta))= \Omega_{\PP^N} (2\theta)$, it induces obviously an exact sequence
  $$ 0 \rightarrow  p_{*}(q^*(J_{\Sigma/G}(\eta) ))  \rightarrow \Omega_{\PP^N} (2\theta) \rightarrow p_{*}(q^*(O_{\Sigma}(\eta)) ).$$
   Since   $p_{*}(q^*(O_{\Sigma}(\eta))$ is a torsion free $O_{\PP^N}$-module of rank-$1$, free outside $X(\Sigma)$  there exists  a positive number $k$ and a sheaf of ideals $J \subset O_{\PP^N}$ such that we have an exact sequence
$$ 0 \rightarrow  p_{*}(q^*(J_{\Sigma}(\eta) ))  \rightarrow \Omega_{\PP^N} (2\theta) \rightarrow J(k\theta)\rightarrow 0. \ \ \ (*)$$
%As we have seen, since $\Sigma$ is Cohen-Macaulay and irreducible,
 It is clear that $J \subset O_{\PP^N}$ is the sheaf of ideals of a scheme with support in $X(\Sigma)$.

\begin{defi} 1) We define the fundamental  scheme $X(\Sigma)$ of $\Sigma$ as the subscheme of $\PP^N$ whose ideal is $J$.
From now we denote this ideal by $J_{X(\Sigma)/\PP^N}$.

2) We define the number $k$ as the secant index of the congruence  $\Sigma$.

\end{defi}

As the reader understand, it is indeed possible to introduce the notion of fundamental scheme without assuming that $\Sigma$ is irreducible. This is not our choice in this short paper.

The notion of secant index in justified  by the coming theorem. 
From the exact sequence $(*)$, we keep particularly in mind the surjective map 
$\Omega_{\PP^N} (2\theta) \rightarrow J_{X(\Sigma)/\PP}(k\theta)\rightarrow 0$.
%and the induced homomorphism  $H^0(\Omega (2\theta) ) \rightarrow H^0(J(k\theta))$
Its interpretation is the key to the next theorem.

\begin{theo}  
 \label{main1}
 
Let $\Sigma$ be a Cohen-Macaulay, order $1$, irreducible congruence of lines.

 1) $q^{-1}(\Sigma)$ is the blowing up of ${\PP}^{N}$ along the fundamental scheme $X(\Sigma)$.

 2) if $k$ is the secant index of $\Sigma$, then $q^*O_{\Sigma} (\eta) = O_{q^{-1}(\Sigma)}(k\theta-E)$, 
where $E$ is the inverse image of $X(\Sigma)$ in the blowing up.

 3) $L \in \Sigma$ if and only if $L$ is a $k$-secant line to the fundamental scheme $X(\Sigma)$.

 4) The image of the composite map $\Lambda^2 V = H^0(\Omega_{\PP^N} (2\theta) ) \rightarrow H^0(J(k\theta))$ is a linear system
of hypersurfaces of degree $k$ defining the map $q^{-1}(\Sigma) \rightarrow \Sigma \subset G(2,V) \subset \PP(\Lambda^2 V)$. 

 5) The linear system cut on $X(\Sigma)$ by hypersurfaces of degree $k-2$ is not complete, i.e.  $H^1(J_{X(\Sigma)/\PP}(k-2)) \neq 0$.
\end{theo}

\begin{proof} From the exact sequence $(*)$ we deduce immediately  $1$) and $2$).

%surjective map $\Omega_{\PP^N} (2\theta) \rightarrow J_{X(\Sigma)/\PP}(k\theta)$ proves that the blowing up of $\PP$ along $X(\Sigma)$ is $q^{-1}(\Sigma)$
%embedds in ${\mathcal  I} = {\PP}_{{\PP}^{N}}( \Omega_{\PP^N}(2\theta))$. This proves $1$) and $2$).

To prove $3$), consider  $\{L\}\in \Sigma$. %and $q^{-1}(\{L\}) \subset I_{\Sigma} \simeq L \subset \PP$, 
There is a scheme isomorphism
$$q^{-1}(\{L\}) \cap E \simeq L \cap X(\Sigma).$$
Assume that $L \subsetneq X(\Sigma)$.
Since $O_{q^{-1}(\Sigma)}(E) = O_{q^{-1}(\Sigma)}(k\theta-\eta)$, it is clear that $L$ is a $k$-secant line to $X(\Sigma)$.
Conversely, if $L$ is a $k$-secant line to $X(\Sigma)$ (not contained in $\Sigma$), then $p^{-1}(L) \subset q^{-1}(\Sigma)$ is $k$-secant to $E$, hence it  is contracted by the line bundle $O_{q^{-1}(\Sigma)}(\eta) = O_{q^{-1}(\Sigma)}(k\theta - E)$.  Consequently $\{L\} \in \Sigma $.  

$4$)  is clear.  

To prove $5$)  we intend to show that the map $$H^1(\Omega_{\PP^N}  ) \rightarrow H^1(J_{X(\Sigma)/\PP}((k-2)\theta)$$ is non zero.
Let $\{L\}\in \Sigma$ be general. Then $L$ is a $k$-secant line to $X(\Sigma)$ not contained in $X(\Sigma)$. Consequently,
$J_{X(\Sigma)/\PP}(k\theta) \otimes O_L = O_L \oplus T$ (where $T$ is a torsion (finite) sheaf on $L$).  From the construction we get a surjective map
$$\Omega_{\PP^N} (2\theta) \otimes O_L = O_L \oplus (N-1)O_L(\theta) \rightarrow J_{X(\Sigma)/\PP}(k\theta) \otimes O_L\simeq O_L\oplus T.$$
This shows that the composite map
$$\CC = H^1(\Omega_{\PP^N}  ) \rightarrow H^1(O_L(-2\theta) \oplus (N-1)O_L(-\theta)) \rightarrow H^1(J_{X(\Sigma)/\PP}((k-2)\theta) \otimes O_L)$$
is not zero. Since it factorizes through $H^1(J_{X(\Sigma)/\PP} ((k-2)\theta))$, we are done.
\end{proof}

\begin{rem} Since $o(\Sigma)=1$, a general line of $\Sigma$ is not contained in $X(\Sigma)$.

\end{rem}

The following remark and the question it brings up are obviously of interest.

\begin{rem} Dualizing the exact sequence $(*)$, we notice that the  fundamental scheme of a Cohen-Macaulay, irreducible congruence of order $1$ is the zero locus of a  section of $\Omega_{\PP^N}^\vee (k-2)$ (where $k$ is the secant index of the congruence). 
 \end{rem}
  \begin{que}
Which are the sections of  $\Omega_{\PP^N}^\vee (k-2)$ whose zero locus  is the fundamental scheme of a Cohen-Macaulay, irreducible congruence of order $1$ with secant index $k$ ?
 \end{que}
 
 It is clear that  a section of $\Omega_{\PP^N}^\vee (k-2)$ defines an embedding of the blowing up $\tilde{\PP^N}$ of $\PP^N$ along the zero locus of the section in the incidence variety. Its image in $G(1,N)$ is a congruence of order $1$ if and only if 
 
$$(L,x) \in \tilde{\PP} \Leftrightarrow (L,y) \in \tilde{\PP}, \ \     \forall y \in L.$$

The following example needs no comment. 
 
 \begin{exa} The zero locus of a section of $\Omega^\vee(-1) = \Omega^\vee (1-2)$ is a point. 
 This point, with its reduced structure, is the fundamental scheme of the congruence formed by lines through it. The secant index of this congruence is $1$.
  \end{exa}

We describe briefly the fundamental scheme and the secant index for all the examples discussed earlier. We follow the same organization in three different groups. It is important to note
immediately that the secant index of a congruence of $k$-secant lines to a smooth variety $X$ is not necessarily $k$
and the fundamental scheme of the congruence is not necessarily $X$. 
The description of the fundamental scheme in each of the coming examples makes this (as well as the fact that the secant index is a multiple of $k$) clear. All the congruences studied in the following examples are Cohen-Macaulay and irreducible (in some cases it is obvious, but not in all cases).

\begin{exas}

- The fundamental scheme of the congruence $\Sigma$ of $2$-secant lines to a twisted cubic $C$ is the first infinitesimal neighborhood of $C$, in other words  $J_{X(\Sigma)/ \PP}=J^2_{C/ \PP}$. The secant index of this congruence is $4$. For $x\in C$, we have
$\Sigma(x)=\PP^1$ and  $O_{q^{-1}(\Sigma)}(\eta) \otimes O_{\Sigma(x)} = O_{\PP^1 }(2)$.

- The fundamental scheme of the congruence $\Sigma$ formed by the lines joining two skew space lines $L_1$ and $L_2$ is
$L_1 \cup L_2$, i.e.  $J_{X(\Sigma)/\PP}= J_{(L_1 \cup L_2)/\PP}  = J_{L_1/\PP} \cap J_{L_2/\PP}$. The secant index of this congruence is $2$.  For $x \in L_1 \cup L_2$, we have $\Sigma(x)=\PP^1$ and $O_{q^{-1}(\Sigma)}(\eta) \otimes O_{\Sigma(x)} = O_{\PP^1 }(1)$.

- The fundamental scheme of the congruence  of  $2$-secant lines to a normal  rational ruled surface (without exceptional line)  $S \subset \PP^5$ is  a multiple structure of order $4$ on $S$, containing strictly the first infinitesimal neighborhood  of $S$. The secant index is $4$.

More precisely, there is an exact sequence $0 \rightarrow J_{X(\Sigma)} \rightarrow J_{S/\PP^5}^2 \rightarrow \mathcal{L}^2$, where 
$\mathcal{L}$ is the quotient of the conormal  bundle of $S$ defined by the family of $\PP^3$ tangent to $S$ along a line. 

\end{exas}

Next we come back to congruences of $3$-secant lines. Note that we get a secant index $3$ in one case and a 
 secant index $9$ in the other case. This is well explained by the description of the fundamental scheme.

\begin{exas}

- The fundamental scheme of the congruence of $3$-secant lines to a projected Veronese surface (in $\PP^4$) is the projected Veronese surface itself. The secant index is $3$.

- The ideal  of the fundamental scheme of the congruence of $3$-secant lines to a Bordiga surface  $B \subset \PP^4$ is
$J_{B/\PP^4}^3 \cap J_{P_1/\PP^4} \cap ... \cap J_{P_9/\PP^4}$ (where $(P_i)_{i=1}^9$ are the $9$ "parasitic" planes cutting a plane cubic curve in $B$). The secant index is $9$.

\end{exas}

Finally, we describe the secant index and the fundamental scheme for two examples of congruence  of $4$-secant lines.

\begin{exas}

- The fundamental scheme of the congruence of $4$-secant lines to a Palatini  $3$-fold (in $\PP^5$) is the Palatini $3$-fold itself. 
The secant index is $4$.

- The ideal of the fundamental scheme of the congruence of $4$-secant lines to a scroll $Sc \subset \PP^5$ over a $K_3$ surface is 
$J_{Sc/\PP^5}^2$. The secant index is $8$.

\end{exas}

Considering these examples, we note that the fundamental scheme is smooth (and the secant index is what it should) in the following cases:

- the congruence,  of lines passing through a point in $\PP^N$ (secant index $1$), 

- the congruence of lines joining $2$ skew lines in $\PP^3$ (secant index $2$),

- the congruence of $3$-secant lines to a projected Veronese surface in $\PP^4$ (secant index\nobreak $3$),

- the congruence of $4$-secant lines to a Palatini $3$-fold in $\PP^5$  (secant index $4$).

To conclude this section, we describe a particular (and well known) family of smooth, order $1$, congruences of lines   with smooth fundamental scheme and secant index $2$.

\begin{pro} Let $V=V_1 \oplus V_2$ be a decomposed complex vector space. The surjective homomoprhism
$\Lambda^2 V \rightarrow V_1\otimes V_2$ induces an isomorphism 
$$\PP(V_1) \times \PP(V_2) \simeq G(2,V)\cap \PP(V_1\otimes V_2).$$

The smooth congruence $\PP(V_1) \times \PP(V_2)$ so defined has order $1$,  it parametrizes the lines joining
$\PP(V_1)$ and $\PP(V_2)$ in $\PP(V)$.

The fundamental scheme of the congruence is the smooth disconnected union $\PP(V_1)  \cup \PP(V_2)$,
except if there exists an $i$ such that $V_i$ has dimension $1$,
in which case the fundamental locus is the point $ {\PP}(V_i)$.

The secant index of the congruence is $2$, except when the fundamental locus is a point, in which case the secant index is $1$.
\end{pro}

The proof is left to the reader.

\section{Cohen-Macaulay, Order $1$,  irreducible congruences of lines with smooth fundamental scheme}

We begin with an almost obvious result.

\begin{pro} If $\Sigma \subset G(1,N)$ is a Cohen-Macaulay, order $1$, irreducible  congruence of lines with smooth fundamental scheme $X(\Sigma)$, then
$\Sigma(x) \subset G(1,N)$ is a linear space for all $x \in X$. 

To be precise,  for $x \notin X(\Sigma)$ then $\Sigma(x) = \PP^0$;  
 for $x \in X(\Sigma)$ then $\Sigma(x) = \PP^{c-1} \subset G(1,N)$ where $c$ is the codimension in $\PP^N$ of the connected component of $X(\Sigma)$ containing $x$.
\end{pro}

\begin{proof} This is a clear consequence of the exact sequence $(*)$.  Indeed, the surjective map
$\Omega_{\PP^N}(2\theta) \rightarrow J_{X(\Sigma)/\Bbb P^N}(k\theta)$  induces for all $x \in X(\Sigma)$ a surjective map
$$(\Omega_{\PP^N}(2\theta)) (x) \rightarrow (N^\vee_{X(\Sigma)/\PP^N}(k\theta))(x), $$ hence an embedding   
$$\Sigma(x)=\PP(N^\vee_{X(\Sigma)/\PP^N}(x)) \simeq \PP^{c-1}  \subset \PP^{N-1}(x)$$

\end{proof}

\begin{defi}  A 
congruence $\Sigma \subset G(2,V)$ is linear if it is cut out  in $G(2,V)$ (scheme theoretically, but not
necessarily properly) by a linear subspace of the PlŸcker space
$\PP(\wedge^2 V)$.
\end{defi}

It is clear that the order of a linear congruence is either $0$ or $1$.

\begin{rem}  

 A congruence of lines $\Sigma$ in ${\PP}^2$ is linear if and only if there exists $x\in {\PP}^2$ such that $\Sigma ={\PP}^1(x)$ parametrizes the lines through $x$.
\end{rem}

This is obvious. The case of linear congruences of order $1$ in ${\PP}^3$ is almost as easy to describe.

\begin{pro}  $\Sigma$ is a linear congruence of order $1$ of lines in ${\PP}^3$ if and only if one of the three following conditions is verified:

1) there exists a point $x \in  {\PP}^3$ such that $\Sigma = {\PP}^2(x)$ parametrizes the lines through $x$, 

2) there exist two skew lines $L_1,L_2 \subset {\PP}^3$ such that $\Sigma$ parametrizes the lines joining $L_1$ and $L_2$,

3) there exists $x\in H \subset {\PP}^3$ such that $\Sigma = {\PP}^2(x) \cup  H^*$ (where $H$ is a plane and $H^*$ the dual plane)

\end{pro}

\begin{proof} Indeed $G(1,3)$ is a quadric in ${\PP}^5$, hence a linear congruence will be cut out by $2$ or $3$ hyperplanes. If the congruence $\Sigma$ is cut out by $3$ hyperplanes, the congruence is a plane. Since the lines in a plane form a congruence of order $0$, only case $1$ can occur. 

 If  the congruence is cut by a pencil of hyperplanes,  this pencil  contains two special linear complexes. If the two corresponding lines are disjoint we are in case $2$, if they intersect (in a point $x$), $3)$ holds. %Please note that in this last case, the fundamental scheme $H$ has an embedded  component in $\{x\}$. Hence $X(\Sigma)$ is not smooth!
 
 \end{proof}
 
Note here that the congruence $\Sigma = {\PP}^2(x) \cup  H^*$ described in $3)$ is Cohen-Macaulay but (obviously) not irreducible. It is in fact the union of a smooth, linear, irreducible congruence of order $1$ and a smooth, linear, irreducible congruence of order $0$. This example explains why  we prefer irreducible congruences.

The following question was raised by Fyodor Zak.

\begin{que}  Are the two following conditions equivalent ?

1) $\Sigma $ is a linear congruence,

2) For every $x\in {\PP}^N$, the scheme  $\Sigma(x)$ is a linear subspace of ${\PP}^{N-1}(x)$.
\end{que}

\vspace{0,3cm}

Time has come to state and prove the main theorem of this paper.

\begin{theo}

 \label{main2}
Let  $\Sigma \subset G(1,N)$ be an order $1$, Cohen-Macaulay, irreducible congruence of lines with smooth fundamental scheme $X( \Sigma )$ and secant index $k$.

1) If $k\leq 2$, the fundamental locus is either a point ($k=1$) or a union of $2$ complementary linear spaces ($k=2$).

2)  If $k \geq  3$, then $X(\Sigma)$ is connected and $k=(N-1)/(c-1)$ where $c$ is the codimension of $X(\Sigma)$ in $\PP$.

3) $K_{\Sigma}=O_{\Sigma}(-c)$.

4) 
The linear system cut out on $X(\Sigma)$ by hypersurfaces of degree $k-2$  is not complete, i.e.  $H^1(J_{X(\Sigma)/\PP}(k-2)) \neq 0$.

The scheme $D= \{x\in X(\Sigma), \ \  \Sigma (x) \subset T_{X,x}\}$  is the zero variety of a section of $O_X(k-2)$ not cut out by a hypersurface  of degree $k-2$. 

Its inverse image in the divisor $E \subset  q^{-1}({\Sigma})$ is the ramification locus of the finite  (degree $k$) map  $E \rightarrow \Sigma $.

\end{theo}  

\begin{proof}  The proof of $1)$ is straightforward. 

 To prove $2)$, note that if $X(\Sigma)$ is not connected then it must have two connected components such that the lines of $\Sigma$
 join the two components. But the lines parametrizing the join  of  two varieties form a family of dimension at most $N-1$. It has to be the congruence $\Sigma$, and it implies that $X(\Sigma)$ is the union of two linear spaces and $k=2$. A contradiction.
 
 From the general projection theorem (see \cite{Gr-Pe}) we know that if the $k$-secant lines to a connected  smooth variety in $\PP^N$ form a congruence of lines, then $k=(N-1)/(c-1)$. This proves  $2)$.

$3)$ is proved by computing twice the canonical line bundle  $K_{q^{-1}(\Sigma)}$. On the one hand 
$q^{-1}(\Sigma)=\PP_{\Sigma}(Q\mid \Sigma)$ and this implies $$K_{q^{-1}(\Sigma)}=q^{*}K_\Sigma \otimes O_{q^{-1}(\Sigma)}(\eta - 2\theta).$$ On the other hand   $q^{-1}(\Sigma)$ is the blowing up of $\PP$ along $X(\Sigma)$ and this proves $$K_{q^{-1}(\Sigma)} = O_{q^{-1}(\Sigma)}(-(N+1)\theta + (c-1)E).$$ Since
$O_{q^{-1}(\Sigma)}(E)=O_{I_{\Sigma}}(k\theta - \eta)$  we find
$$q^{*}K_\Sigma(\eta - 2\theta) = O_{q^{-1}(\Sigma)}(-2\theta - [N-1-(c-1)k]\theta - [c-1]\eta)$$
which proves $4)$ (by using $2)$).

 $4)$ The first assertion has already been proved without assuming that $X(\Sigma)$ is smooth.
 
Concerning the second assertion, we note first that an elementary computation proves that the ramification $K_E\otimes q^{*}K_\Sigma^\vee$ of the generically finite map $E \rightarrow \Sigma$ is  a section of $O_E((k-2)\theta)$. We claim that the ramification is not cut out by a hypersurface of degree $k-2$ of $\PP$. Indeed, following an idea of F. Han, we consider the relative Euler complex 
$$ 0 \rightarrow O_{q^{-1}(\Sigma)}(\eta - 2\theta) \rightarrow  Q(-\theta)  \rightarrow O_{q^{-1}(\Sigma)} \rightarrow 0$$
of the bundle map $q^{-1}(\Sigma) \rightarrow \Sigma$. It fits in the following commutative diagram, with exact rows and columns:
%{\tiny
$$\begin{CD}
@. @.  0@. 0 @.\\
@. @. @VVV @VVV \\
@.  @.  O_{q^{-1}(\Sigma)}( \eta - k\theta) @= O_{q^{-1}(\Sigma)}( \eta - k\theta) @.\\
@. @. @VVV @VVV \\
0@ >>>    O_{q^{-1}(\Sigma)}( \eta - 2\theta)  @>>> Q(-\theta) @>>> O_{q^{-1}(\Sigma)} @ >>> 0\\
@. @|  @VVV @VVV \\
0  @ >>>  O_{q^{-1}(\Sigma)}( \eta - 2\theta) @ >>> J_{R/q^{-1}(\Sigma)}((k-2)\theta) @>>>  O_E @>>> 0\\
@.@. @VVV @VVV \\
 @.    @. 0@.  0 @.\\
\end{CD}$$
%}
where $J_{R/q^{-1}(\Sigma)}$ is the ideal of the ramification in $q^{-1}(\Sigma)$. This diagram proves that 
$H^0(J_{R/q^{-1}(\Sigma)}((k-2)\theta))=0$ and confirms that $J_{R/E}((k-2)\theta)=O_E$. Since
the divisor $R\subset E$ is the inverse image of the divisor $D\subset X(\Sigma)$, we are done.

\end{proof}

\begin{rem} The  surjective homomorphism
$\Omega _{\PP^N}(2\theta)\otimes O_{X(\Sigma)} \rightarrow N_{X(\Sigma)/\PP^N}^\vee(k\theta)$ 
defines a map  from  $X(\Sigma)$ to the Fano variety of linear spaces of dimension $c-1$ in $\Sigma$ (Zak's map). 

This map is easily proved to be an isomorphism (communicated by F. Zak) but this is not the subject of this paper.

\end{rem}

The above theorem comes with two natural conjectures that I failed to prove (very irritating!).

\begin{conjs} If $\Sigma$ is as in the theorem, then 

1)  $\Sigma \subset G(2,V) \subset\PP(\Lambda^2 V)$ is linearly normal (see \cite{Il-Ma})

2) $X(\Sigma)$ is $k$-regular (Castelnuovo-Mumford regularity).

\end{conjs}

These two conjectures are perhaps justified,  more probably explained,  by  the classification of all  order $1$ congruences with smooth fundamental scheme and secant index $\leq 3$ and the description of the  two known examples with secant index $4$. 

%Note that since $K_{\Sigma}=O_{\Sigma}(-c\eta)$,  

\begin{theo}   (Classification Theorem) Let $\Sigma \subset G(1,N)$ be an order $1$, Cohen-Macaulay, irreducible congruence of lines with smooth fundamental scheme.
Let $k$ be the secant index of $\Sigma$.

1) If $k=1$, then $\Sigma = {\PP}^{N-1}(x)$, with 
$x \in {\PP}^N$ and  $X(\Sigma) = \{ x \}$. 

The ramification divisor is empty and cut out by a non zero section of $O_{\{x\}}(-1)$

2) If $k=2$, %then $N=2n+1$. If if ${\PP}^N={\PP}(W)$, with $rk(W)=2n+2$ 
there exists a decomposition $W=V_1\oplus V_2$ with $\dim_{\CC} (V_i) \geq 2$ and 
$\Sigma ={\PP}(V_1)\times {\PP}(V_2) = {\PP}(V_1\otimes V_2) \cap G(1,N) \subset {\PP}(\Lambda V)$ and $X(\Sigma) = {\PP}(V_1) \cup  {\PP}(V_2)$.

The ramification divisor is empty and cut out by an everywhere non zero section of $O_{X(\Sigma)}$.

3) If $k=3$,  $\Sigma $ is the congruence  of $3$-secant lines to a projected Severi variety $S=X( \Sigma)$.

The ramification divisor $D$ is cut out in $ X(\Sigma)$ by a "virtual  hyperplane", i.e. $O_{X(\Sigma)}(D) = O_{X(\Sigma)}(\theta)$ but
$D$ is not cut out by an hyperplane in $ X(\Sigma)$.

\end{theo}

\begin{proof} $1)$ and $2)$ are obvious from our main theorem. 

From the same theorem  we see that if $k=3$ then $N-1=3(c-1)$ and
$X(\Sigma)$ is not linearly normal. By Zak's celebrated classification of Severi varieties (\cite{Zak}) we see that $X(\Sigma)$ has to be one of the projected Severi varieties and $\Sigma$ the variety formed by its $3$-secant lines. Note that Iliev and Manivel have proved that 
$\Sigma$ is indeed linearly complete (hence projectively Cohen-Macaulay) in this case.
\end{proof}

Next we recall  the two known order $1$ congruences with smooth fundamental scheme and secant index $4$.

\begin{pro} There exist  two known congruences  with secant index $4$. %The Palatini congruences.

a) $\Sigma \subset G(1,5)$ is formed of the $4$-secant lines to its fundamental scheme, the Palatini $3$-fold $X( \Sigma) \subset \PP^5$.

The ramification locus is cut out in $X( \Sigma)$ by a "virtual quadric".

b)  The second  congruence  $\Sigma \subset G(1,9)$ is formed of the $4$-secant lines to its fundamental scheme, the second Palatini variety (sometimes described under another name) $X(\Sigma)$ a $6$ dimensional smooth variety, cut out by the maximal pfaffian ideal
of a general form $\tau \in H^0(\Lambda^2 \Omega_{\PP^9}(2))$.

\end{pro}

\begin{proof}

We have already seen the case of the $4$-secant lines to a Palatini $3$-fold (which as we know is not quadratically normal).

For b), consider  a general $\tau \in \Lambda^3 V= H^0(\Lambda^2 \Omega_{\PP^9}(2))$. It induces a map $V^\vee \rightarrow \Lambda ^2 V$ which cuts out a linear space in the Plucker space $\PP( \Lambda ^2 V)$. This linear space cuts (improperly)  a linear congruence $\Sigma \subset G(1,9)$ whose fundamental scheme  in $\PP^9$ is the the variety cut out by the maximal pfaffian ideal of
 $\tau$.

\end{proof}

 We conclude with a conjecture  (relating our classification to Hartshorne low codimension conjecture).

 \begin{conj} The congruences listed in the theorem and the proposition form the exhaustive list of Cohen-Macaulay, order $1$,
 irreducible congruences with smooth fundamental scheme.

 \end{conj} 
We recall here Hartshorne's celebrated conjecture for low codimension smooth varieties:  if $N>3c$, a smooth variety of codimension $c$ in $\PP^N$ is a complete intersection.

From our main theorem, we know that if $\Sigma \subset G(1,N)$ has smooth fundamental scheme, then $X(\Sigma)$ is not projectively normal, hence not a complete intersection.

An elementary computation shows that if Hatshorne's conjecture is true, the only possible unknown Cohen-Macaulay, order $1$, irreducible congruences with
smooth fundamental scheme would have the following invariants:

- $k=4$ and  $N=5$ or $N=9$, precisely the invariants of  the congruences of $4$-secant lines to the two Palatini varieties,

- $k=5$ and $N=6$, in other words $X(\Sigma)$ would be a smooth codimension $2$ variety in $\PP^6$ not cubically normal.

\vspace{0,5cm}

Christian Peskine, Universit{\'e} Pierre et Marie Curie

Institut de MathŽmatiques de Jussieu-Paris Rive Gauche

christian.peskine@imj-prg.fr


\begin{thebibliography}{FMV}


\bibitem[Gr-Pe]{Gr-Pe} Gruson, L.; Peskine, C. On the smooth locus of aligned Hilbert schemes, the k-secant lemma and the general projection theorem. Duke Math. J. 162 (2013), no. 3, 553Ð578.

\bibitem[Han]{Han} Han, F.  Duality and quadratic normality. Rend. Istit. Mat. Univ. Trieste. Vol 47 (2015)

\bibitem[Il-Ma]{Il-Ma} Iliev, A.; Manivel, L. Severi varieties and their varieties of reductions. J. Reine Angew. Math. 585 (2005), 93Ð139.

\bibitem[Zak]{Zak} Zak, F. L. Tangents and secants of algebraic varieties. Translated from the Russian manuscript by the author. Translations of Mathematical Monographs, 127.
American Mathematical Society, Providence, RI, 1993.


\end{thebibliography}
\end{document}